\documentclass[11pt,a4paper,twoside]{amsart}
\usepackage{amssymb,amsmath,amsthm,graphicx}
\theoremstyle{plain}
\newtheorem{theorem}{Theorem}[section]

\newtheorem{definition}[theorem]{Definition}
\newtheorem{lemma}[theorem]{Lemma}
\newtheorem{corollary}[theorem]{Corollary}
\newtheorem{proposition}[theorem]{Proposition}

\theoremstyle{remark}
\newtheorem{remark}[theorem]{Remark}

\usepackage{hyperref}

\numberwithin{equation}{section}

\newcommand{\C}{\mathbb{C}}
\newcommand{\R}{\mathbb{R}}

\newcommand{\I}{\infty}

\newcommand{\norm}[1]{\left\lVert #1\right\rVert}

\newcommand{\Jbr}[1]{\left\langle #1 \right\rangle}

\def\({\left(}
\def\){\right)}
\def\<{\left\langle}
\def\>{\right\rangle}
\def\le{\leqslant}
\def\ge{\geqslant}

\def\d{{\partial}}

\newcommand{\eps}{\varepsilon}

\newcommand{\rre}{\mathbb{R}}
\newcommand{\pt}{\partial}

\begin{document}
\title[Scattering for the generalized KdV equation]
{Scattering problem for the generalized\\
Korteweg-de Vries equation}

\author{Satoshi Masaki}
\address{Department of Mathematics, 
Faculty of Science, Hokkaido University,
Kita 10, Nishi 8, Kita-Ku, Sapporo, Hokkaido, 060-0810, Japan}
\email{masaki@math.sci.hokudai.ac.jp}

\author{Jun-ichi Segata}
\address{Faculty of Mathematics, Kyushu University, 
Fukuoka, 819-0395, Japan}
\email{segata@math.kyushu-u.ac.jp}



\maketitle

\begin{abstract}
In this paper we study the scattering problem 
for the initial value problem of the generalized 
Korteweg-de Vries (gKdV) equation. 
The purpose of this paper is to achieve two 
primary goals. Firstly, we show small data scattering 
for (gKdV) in the weighted Sobolev 
space, ensuring the initial and the asymptotic states 
belong to the  same class. 
Secondly, we introduce two equivalent characterizations 
of scattering in the weighted Sobolev space.
In particular, this involves the so-called conditional 
scattering in the weighted Sobolev space.
A key ingredient is incorporation of the 
scattering criterion for (gKdV) in the Fourier-Lebesgue 
space by the authors \cite{MS} into the the scattering 
problem in the weighted Sobolev space.
\end{abstract}

\section{Introduction}

In this paper we study the scattering problem 
for the generalized Korteweg-de Vries (gKdV) equation
\begin{equation}\label{gKdV}
	\pt_tu+\pt_x^3u=\mu\pt_x(|u|^{2\alpha}u),
	\qquad t,x\in\rre
\end{equation}
under the initial condition
\begin{equation}\label{IC}
	u(0,x)=u_{0}(x), \qquad\qquad x \in \R,
\end{equation}
where $u:\rre\times\rre\to\rre$ is an unknown function, 
$u_{0}:\rre\to\rre$ is a given function, 
and $\mu\in\rre\backslash\{0\}$ and $\alpha>0$ 
are constants. We call that (\ref{gKdV}) is defocusing if 
$\mu>0$ and focusing if $\mu<0$. 
Equation (\ref{gKdV}) is a generalization of 
the Korteweg-de Vries equation which models long waves 
propagating in a channel \cite{KV} and 
the modified Korteweg-de Vries equation which 
describes a time evolution for the  curvature of 
certain types of helical space curves \cite{L}. 

Equation (\ref{gKdV}) has the following conservation 
laws: If $u(t)$ is a solution to (\ref{gKdV}) on the time 
interval $I$ with $0\in I$, 
then, $u(t)$ has conservation of the mass
\begin{eqnarray}
	M[u(t)] := \frac12 \norm{u(t,\cdot)}_{L^2}^2=M[u_0] 
	\label{mass}
\end{eqnarray}
and conservation of the energy
\begin{eqnarray}
	E[u(t)] := \frac12 \norm{\d_x u(t,\cdot)}_{L^2}^2 + \frac{\mu}{2\alpha+2} 
	\norm{u(t,\cdot)}_{L^{2\alpha+2}}^{2\alpha+2}=E[u_0]
	\label{energy}
\end{eqnarray}
for any $t\in I$.


We take the initial data $u_0$ from
the weighted Sobolev space $H^1\cap H^{0,1}$, 
where $H^1$ is the usual Sobolev space and 
$H^{0,1}$ is the weighted $L^2$ space defined by 
\begin{eqnarray*}
	H^{0,1}=H^{0,1}(\rre):=\{f\in L^2(\rre)\ ;\ 
	\norm{f}_{H^{0,1}} =\|\langle x\rangle f\|_{L^{2}}<\infty\}
\end{eqnarray*}
with $\langle x\rangle=\sqrt{1+|x|^2}$. 
By the Sobolev embedding, one sees that
the weighted space $H^1 \cap H^{0,1}$ is embedded into
$L^r \cap \hat{L}^r$ for any $r\in [1,\infty]$,
where $\hat{L}^r$ is the Fourier-Lebesgue space 
defined for $1\le r \le \infty$ by 
\begin{eqnarray}
	\hat{L}^r=\hat{L}^r(\rre):=\{f\in{{\mathcal S}}'(\rre)\ ;\ 
	\norm{f}_{\hat{L}^r} =\|\hat{f}\|_{L^{r'}}<\infty\}
	\label{hatL_p}
\end{eqnarray}
and $r'$ denotes the H\"older conjugate of $r$. 

The purpose of this paper is to achieve two 
primary goals. Firstly, we show small data scattering 
for (\ref{gKdV})-\eqref{IC} in the weighted Sobolev 
space, ensuring the initial and the asymptotic states 
belong to the  same class. 
Secondly, we introduce two equivalent characterizations 
of scattering in the weighted Sobolev space.
In particular, this involves the so-called conditional 
scattering in the weighted Sobolev space.

There are many results on 
the small data scattering problem for (\ref{gKdV}). 
Strauss \cite{St} proved that if $\alpha>(3+\sqrt{21})/4$,  
and $u_0\in L^{(2\alpha+2)/(2\alpha+1)}$, $\pt_xu_0\in L^2$ 
are sufficiently small, then the solution to (\ref{gKdV}) 
is global and scatters in $H^1$. 
Ponce and Vega \cite{PV} have shown a similar 
scattering result for $\alpha>(5+\sqrt{73})/8$. 
Christ and Weintein \cite{CW} improved their results 
to $\alpha>(19-\sqrt{57})/8$. 
Furthermore, Hayashi and Naumkin  
extended their results to $\alpha\ge1$, where they 
proved an usual scattering for (\ref{gKdV}) when $\alpha>1$ 
\cite{HN1} (see also C\^{o}te \cite{C} for construction of 
large data wave operator) and a modified scattering for 
$\alpha=1$ \cite{HN2,HN3,HN4} 
(See also Harrop-Griffiths \cite{B}, 
Germain, Pusateri and Rousset \cite{GPR}, 
Correia, C\^{o}te, and Vega \cite{CCV} 
for other approaches). 
In those results, the classes of the initial states and 
the asymptotic states are different. 

Form the physical 
perspective, it is natural 
that the initial and the asymptotic states belong 
to the same class.
For this direction, Kenig, Ponce and Vega \cite{KPV} 
proved the small scattering of (\ref{gKdV}) 
in the scaling critical space $\dot{H}^{s_{\alpha}}$ 
for $\alpha\ge2$, where $s_{\alpha}:=1/2-1/\alpha$ 
is a scaling critical exponent (see also Strunk \cite{Strunk}). 
Since the scaling critical exponent $s_{\alpha}$ is
negative in the mass-subcritical case $\alpha<2$, 
the scattering of (\ref{gKdV}) in the scaling critical 
space $\dot{H}^{s_{\alpha}}$ becomes rather a difficult problem. 
Tao \cite{T1} proved 
global well-posedness and scattering for small data for (\ref{gKdV}) 
with the quartic nonlinearity $\mu\pt_{x}(u^{4})$ 
in $\dot{H}^{s_{3/2}}$.  
Later on, Koch and Marzuola \cite{KM} simplified Tao's 
proof and extended his result to a Besov space 
$\dot{B}^{s_{3/2}}_{\infty,2}$. 
In \cite{MS}, the authors proved small data scattering for 
(\ref{gKdV}) in the framework of the scaling critical 
Fourier-Lebesgue space $\hat{L}^{\alpha}$ for 
$8/5<\alpha\le2$.



For the large initial data, Dodson \cite{D} has shown 
the global well-posedness and scattering in $L^2$ 
for (\ref{gKdV}) with the defocusing and mass-critical 
nonlinearity (i.e., $\mu>0$ and $\alpha=2$) 
by using the concentration compactness argument by 
Kenig and Merle \cite{KenigMerle} and 
the monotonicity formula for (\ref{gKdV}) 
by Tao \cite{T2} (see also Killip, Kwon, Shao and Vi\c{s}an \cite{KKSV} 
for the existence of the minimal non-scattering solution 
for (\ref{gKdV}) with the focusing, mass-critical nonlinearity). 
After that Farah, Linares, Pastor and Visciglia \cite{FLPV} proved 
the global well-posedness and scattering in $H^1$ for (\ref{gKdV}) 
with the defocusing and mass-supercritical 
nonlinearity (i.e., $\mu>0$ and $\alpha>2$) by adapting the 
the concentration compactness argument into $H^1$. 
For the mass-subcritical case $\alpha<2$, the authors 
\cite{MS2,MS3} proved the existence of the minimal non-scattering 
solution for (\ref{gKdV}) with $5/3<\alpha<2$ by applying 
the concentration compactness argument in the 
Fourier-Bourgain-Morrey space. 
Furthermore, Kim \cite{K} proved the conditional 
scattering in the Fourier-Bourgain-Morrey space for (\ref{gKdV}) 
when the nonlinear term is defocusing and mass-subcritical 
with $5/3<\alpha<2$. 
Note that for the case $\alpha=1$, 
it is well-known that (\ref{gKdV}) is completely integrable. 
By using the inverse scattering method, Deift-Zhou \cite{DZ} 
obtained asymptotic behavior in time of solution to 
(\ref{gKdV}) with $\alpha=1$ and without smallness on the 
initial data. 

\subsection{Local well-posedness in a weighted space}
In this paper, we use several notions of a solution to (\ref{gKdV}).
Let $\{V(t)\}_{t\in\rre}$ be a unitary group 
generated by the $-\pt_x^3$.
For an interval $I\subset \R$, we define
\begin{eqnarray}
S(I)&:=&\{u:I\times \R\to\R\ ;\ \|u\|_{S(I)}<\infty\},\label{sn}\\
\|u\|_{S(I)}&:=&\|u\|_{L_x^{\frac52\alpha}(\R;L_t^{5\alpha}(I))}.
\nonumber
\end{eqnarray}

\begin{definition}[a solution to \eqref{gKdV}]
Let $X=\hat{L}^\alpha$, $X=H^1 \cap \hat{L}^\alpha$, or $X=H^1 \cap H^{0,1}$.
For an interval $I \subset \R$, we say a function $u:I\times \R \to \R$ is a $X$-solution on $I$ if $V(-t)u(t) \in C(I;X)$,
$\|u\|_{ S(J)}<\infty$ for any compact $J\subset I$, and the identity
\begin{equation}\label{E:sol}
	V(-t_2)u(t_2) = V(-t_1)u(t_1) 
	+ \int_{t_1}^{t_2} V(-\tau) \pt_x (|u|^{2\alpha}u)(\tau) d\tau
\end{equation}
holds for any $t_1,t_2 \in I$.
\end{definition}
Due to the modification in the definition of a solution, a natural extension of the initial condition \eqref{IC} to an arbitrary time $t_0\in\R$
 is as follows:
\begin{equation}\label{gIC}
	V(-t_0)u(t_0) = V(-t_0)u_0 \in X.
\end{equation}

\begin{remark}
$V(t)$ is an isometry on $\hat{L}^\alpha$ 
or $H^1 \cap \hat{L}^\alpha$.
Hence, $V(-t)u(t) \in C(I;X)$ is equivalent 
to $u(t) \in C(I;X)$ when $X=\hat{L}^\alpha$ 
or $X=H^1 \cap \hat{L}^\alpha$.
Moreover, \eqref{E:sol} is equivalent 
to the validity of the standard Duhamel formula.
Furthermore, \eqref{gIC} is equivalent to 
$u(t_0)=u_0 \in X$.
However, $V(t)$ is not a bounded operator 
from $H^1 \cap H^{0,1}$ to itself 
for any $t\neq0$ and hence
these modifications are essential
in the case $X=H^1 \cap H^{0,1}$.
We also remark that the embedding
\[
	H^1 \cap H^{0,1} \hookrightarrow
	H^1 \cap \hat{L}^{\alpha} \hookrightarrow
	\hat{L}^{\alpha}
\]
holds for any $1\le\alpha\le\infty$. 
Hence, a $H^1 \cap H^{0,1}$-solution is a 
$H^1 \cap \hat{L}^{\alpha}$-solution and 
similarly a $H^1 \cap \hat{L}^{\alpha}$-solution 
is a $\hat{L}^{\alpha}$-solution.
Further, it is known that $\hat{L}^\alpha$-solution 
is unique if $8/5<\alpha<10/3$ 
(see \cite[Theorem 1.2]{MS}).
\end{remark}

Before the scattering problem, 
let us consider the local well-posedness.
It is noteworthy that the local well-posedness 
in $\hat{L}^\alpha$ and $H^1 \cap \hat{L}^\alpha$ 
are already established in \cite{MS}.
We also have the local well-posedness 
in the weighted Sobolev space $H^1 \cap H^{0,1}$.

\begin{theorem}[Local well-posedness in 
$H^1 \cap H^{0,1}$]\label{T:lwp}
The initial value problem \eqref{gKdV} 
under \eqref{gIC} is locally well-posed 
in the weighted Sobolev space $H^1 \cap H^{0,1}$.
More precisely,
suppose that $V(-t_0)u_0 \in H^1 \cap H^{0,1}$ 
for some $t_0 \in \R$.
Then, there exist a interval $I\ni t_0$ and
 a unique $H^1 \cap H^{0,1}$-solution $u(t)$ 
 to \eqref{gKdV} under \eqref{gIC} on $I$
 such that
\[
\|V(-t)u\|_{L^\infty_t(I;H^1_x\cap H^{0,1}_x)}\lesssim 
\|V(-t_0)u_0\|_{H^1_x\cap H^{0,1}_x} 
+ \Jbr{t_0}\|u_0\|_{H^1_x}^{2\alpha+1}.
\]
Moreover, the data-to-solution map $V(-t_0)u_0 \mapsto u$ 
is a continuous map from $H^1 \cap H^{0,1}$ to 
$L^\infty (I;H^1\cap H^{0,1})$.
\end{theorem}

Now, we turn to the global existence of a solution. 
To this end, we introduce the notion of 
the maximal lifespan of a solution.
For a $X$-solution $u(t)$ to \eqref{gKdV} on an interval $I$, 
we define 
\begin{eqnarray*}
T_{\mathrm{max}}&:=&\sup\{T\in \R;\exists u
: X\text{-solution\ to\ (\ref{gKdV}) on }[t_0,T] \},\\
T_{\mathrm{min}}&:=&\inf\{T \in \R;\exists u
:X\text{-solution\ to\ (\ref{gKdV}) on }[T,t_0]\}
\end{eqnarray*}
with a picked $t_0 \in I$.
Note that these quantities are independent of 
the choice of $t_0 \in I$.
Further, we refer 
$I_{\max}=(T_{\mathrm{min}},T_{\mathrm{max}})$
to as the maximal lifespan of a solution $u$.
A solution $u$ on $I_{\max}$ is referred 
to as a maximal-lifespan solution.
We say a solution $u$ is global for positive time direction 
(resp. negative time direction) if $T_{\max}=\infty$ 
(resp. $T_{\min}=-\infty$).

It is obvious from the definition that 
$I_{\max}$ depends on the choice of 
the notion of a solution, i.e., on $X$.
However, those with $X=\hat{L}^\alpha$
and $X=H^1 \cap \hat{L}^\alpha$ coincides each other.
This property, which is called the persistence of 
$H^1$-regularity, implies that if a 
$\hat{L}^\alpha$-solution $u$ satisfies $u(t) \in H^1$ 
at some time in its maximal lifespan 
(as a $\hat{L}^\alpha$-solution) then $u(t)\in H^1$ 
holds in the whole maximal lifespan and further 
$u$ is a $H^1 \cap \hat{L}^\alpha$-solution with 
the same maximal lifespan.
Our next result shows that $I_{\max}$ is also 
the same for $H^1\cap H^{0,1}$-solution.

\begin{theorem}[Blowup alternative]\label{T:bu}
Let $u$ be a maximal-lifespan 
$H^1 \cap H^{0,1}$-solution and 
$I_{\max}=(T_{\min},T_{\max})$ be 
its maximal lifespan as a $H^1 \cap H^{0,1}$-solution.
If $T_{\max}<\infty$ then
\[
	\lim_{T\to T_{\max}-0} \|u\|_{S([t_0,T))} = \infty.
\]
A similar alternative holds for $T_{\min}$.
In particular, $I_{\max}$ is the same as 
those as a $\hat{L}^\alpha$- and 
$H^1 \cap \hat{L}^\alpha$-solution.
\end{theorem}

This property reads as the persistence 
of the boundedness 
$V(-t)u(t) \in H^1 \cap H^{0,1}$ for 
$\hat{L}^\alpha$-solutions.
Due to this property, we use the notation 
$I_{\max}$ without clarifying
the notion of a solution.

\subsection{Main results}
Now, we consider the scattering problem.
We give the definition of scattering in $X$.

\begin{definition}
Let $X=\hat{L}^\alpha$, $X=H^1 \cap \hat{L}^\alpha$, 
or $X=H^1 \cap H^{0,1}$.
We say a $X$-solution $u(t)$ scatters in $X$ 
for positive time direction
if $T_{\mathrm{max}}=+\I$ and 
there exists a unique function $u_{+}\in X$ such that 
\begin{eqnarray}
\lim_{t\to+\infty}\|V(-t)u(t)-u_{+}
\|_{X}=0.\label{scattering}
\end{eqnarray}
The scattering  
for negative time direction is defined by a similar fashion. 
\end{definition}

Our first result is the scattering for small data.

\begin{theorem}[Small data scattering]\label{thm1} 
Let $8/5<\alpha<2$. 
Then there exists $\varepsilon_{0}>0$ 
such that if $u_{0}\in H^1\cap H^{0,1}(\rre)$ 
satisfies $\|u_{0}\|_{H^1\cap H^{0,1}}\le\varepsilon_{0}$, 
then the unique $H^1\cap H^{0,1}$-solution $u$ to (\ref{gKdV}) 
given in Theorem \ref{T:lwp} scatters in 
$H^1\cap H^{0,1}$ for both time directions.
Moreover,
\[
\|V(-t)u\|_{L^\infty (\R;H^1 \cap H^{0,1})} +
\|u\|_{S(\R)}
+\sup_{t\in \R} \Jbr{t}^{\frac13} \|u(t)\|_{L_{x}^{\infty}}
\lesssim  \|u_{0}\|_{H^1 \cap H^{0,1}}.
\]
\end{theorem}
We remark that 
the scattering in $\hat{L}^\alpha$ and 
$H^1 \cap \hat{L}^\alpha$ hold 
with a weaker smallness assumption for $8/5<\alpha<2$. 
More precisely,
for $u_0 \in \hat{L}^\alpha$ if 
\[
\|V(t)u_0\|_{S(\R)} 
+ \| |D_x|^{\frac34-\frac1{2\alpha}} V(t)u_0 
\|_{L^{\frac{20\alpha}{10-3\alpha}}_xL_t^{\frac{10}3} (\R)}
\]
is sufficiently small, then the unique 
$\hat{L}^\alpha$-solution $u(t)$ scatters in $\hat{L}^\alpha$ 
for both time directions.
We emphasize that the smallness of 
$\|u_0\|_{\hat{L}^\alpha}$ is a sufficient condition for this assumption but not a necessary condition.
By the persistence-of-regularity argument, one sees that
if $u_0 \in H^1$ in addition then $u(t)$ scatters in 
$H^1 \cap \hat{L}^\alpha$.
Although Theorem \ref{T:lwp} follows 
by a similar persistence-of-regularity type argument, 
a stronger smallness assumption is required in Theorem \ref{thm1}.

The second main result is 
the two equivalent characterizations of the scattering 
in the weighted Sobolev space.
\begin{theorem}[Scattering criterion]\label{thm2}
Assume $8/5<\alpha<2$. 
Let $u(t)$ be a unique maximal-lifespan $H^1\cap H^{0,1}$-solution of \eqref{gKdV} under \eqref{gIC}.
The following statements are equivalent:
\begin{enumerate}
\item $u(t)$ scatters for positive time direction in $H^1 \cap H^{0,1}$;
\item $u(t)$ is bounded in a weighted norm, i.e., for some $t_0 \in I_{\max}$,
\begin{eqnarray}
\|V(-t)u\|_{L_t^{\infty}H_x^{0,1}([t_0,T_{\mathrm{max}}))}<+\infty.
\label{gb}
\end{eqnarray}
\item There exist $\kappa> \frac{\alpha}{3(\alpha-1)(2\alpha+1)}$ and
$t_0 \in I_{\max}$ such that
\[
	\|u\|_{S([t_0,T_{\max}))} +\sup_{t\in [t_0,T_{\max})} \Jbr{t}^{\kappa}\|u \|_{L^{2(2\alpha+1)}_x} <+\infty.
\]
\end{enumerate}
Further, if one of the above is satisfied then $T_{\max}=\infty$ and
\[
\|V(-t)u\|_{L^\infty ([t_0,\infty);H^1 \cap H^{0,1})} +
\|u\|_{S([t_0,\infty))}
+\sup_{t\in [t_0,\infty)} \Jbr{t}^{\frac13} \|u(t)\|_{L_{x}^{\infty}}
<\infty
\]
for any $t_0 \in I_{\max}$.
The similar statements are true for negative time direction.
\end{theorem}

\begin{remark}
For a $\hat{L}^\alpha$-solution, the boundedness 
$\|u\|_{S([t_0,T_{\max}))} < \infty$
is a necessary and sufficient condition for scattering 
in $\hat{L}^\alpha$ for positive time direction.
The equivalence of (i) and (iii) in Theorem \ref{thm2} 
implies that the additional boundedness condition
\[
	\sup_{t\in [t_0,T_{\max})} \langle t \rangle^{\kappa}
	\|u \|_{L^{2(2\alpha+1)}_x} < +\infty
\]
bridges the gap between scattering 
in $\hat{L}^\alpha$ and in $H^1 \cap H^{0,1}$.
This gap arises due to the weakness of our persistence result. 
A standard persistence-of-regularity argument shows that 
$\|u\|_{S([t_0,T_{\max}))} < \infty$
is also an equivalent characterization of
scattering for positive time direction in 
$H^1 \cap \hat{L}^\alpha$ for $H^1 \cap \hat{L}^\alpha$-solutions.
\end{remark}

We remark that
the implication ``(ii)$\Rightarrow$(i)'' in Theorem \ref{thm2} 
reads as a conditional scattering result.
Indeed, it establishes the scattering under the hypothesis 
of the a priori bound \eqref{gb}.
As mentioned above,
Kim \cite{K} showed a conditional scattering result 
for $5/3<\alpha <2$ under the boundedness in $H^1$ 
and in a Fourier-Bourgain-Morrey space.
Compared with the result, 
Theorem \ref{thm2}
covers a wider range $8/5<\alpha<2$
with a stronger boundedness assumption.

Let us compare the conditional scattering result
Theorem \ref{thm2} with the similar  
results for the mass-subcritical nonlinear Schr\"{o}dinger equation: 
\begin{eqnarray}
\left\{
\begin{array}{l}
\displaystyle{
i\pt_tu+\Delta u=\mu|u|^{2\alpha}u,
\qquad t\in\R,x\in\rre^d,}\\
\displaystyle{u(0,x)=u_{0}(x),
\qquad\qquad\qquad\ \  x\in\rre^d,}
\end{array}
\right.
\label{NLS}
\end{eqnarray}
where $u:\rre\times\rre^d\to\C$ is an unknown function, 
$u_{0}:\rre^d\to\C$ is a given function, 
and $\mu\in\rre\backslash\{0\}$ and $0<\alpha<2/d$  
are constants. 
For (\ref{NLS}) with the defocusing nonlinearity (i.e., $\mu>0$), 
by utilizing the pseudo-conformal transform or pseudo-conformal conservation law, 
it is shown in \cite{T,HT,CaWe,NO} that any $H^{0,1}$-solution
scatters in $H^{0,1}$
when $\alpha \ge 
\alpha_{\mathrm{St}}:=(-d+2+\sqrt{d^2+12d+4})/(4d)$.
As far as the authors know, this kind of transform or conservation law 
are not known for (\ref{gKdV}). 
As for the conditional scattering, Killip, Murphy, Vi\c{s}an 
and the first author \cite{KMMV1,KMMV2} 
proved scattering under the boundedness assumption 
with respect to a scaling critical homogeneous weighted norm 
or to a homogeneous Sobolev norm (see \cite{M1,M2} 
for similar study for $\mu<0$).

\subsection{Outline of the proof}

To investigate the property $V(-t) u(t) \in H^{0,1}$, 
it is convenient to introduce the operator
\[
	J(t):=V(t)xV(-t)=x-3t\pt_x^2.
\]
One strategy is that, we establish a persistence-type property in the weighted Sobolev space by looking at the equation for $Ju$.
This argument works well for the NLS equation \eqref{NLS}.
However, for the generalized KdV equation (\ref{gKdV}), 
the operator $J(t)$ does not work well with the nonlinear term. 
To overcome this difficulty, 
as in Hayashi and Naumkin \cite{HN1,HN2,HN3}, 
we introduce another variable
\begin{eqnarray}
v(t):=J(t)u(t)+3\mu t|u(t)|^{2\alpha}u(t).\label{v}
\end{eqnarray}
Note that if $u(t)$ is a solution to \eqref{gKdV} then one has
$v(t) = (x + 3\partial_x^{-1} \partial_t )u$, at least formally.
We would like to point out that our $v$ does not involve 
an anti-derivative $\partial_x^{-1}$.
A direct computation shows that $v$ solves a KdV-like 
equation
\begin{equation}\label{E:v}
\partial_t v + \partial_x^3 v 
= (2\alpha+1)\mu  |u|^{2\alpha} \partial_x v 
- 2(\alpha-1) \mu |u|^{2\alpha} u.
\end{equation}
It is noteworthy that the equation is written in the integral form and
hence that one can utilize the Strichartz estimates 
to obtain various estimates for $v$.


%

The following notation will be used throughout this 
paper: We use $A\lesssim B$ to denote the estimate 
$A\le CB$ where $C$ is a positive constant. 
$|D_x|^s=(-\pt_x^2)^{s/2}$ and 
$\langle D_x\rangle^s=(I-\pt_x^2)^{s/2}$ denote 
the Riesz and Bessel potentials of order $-s$, 
respectively. For $1\le p,q\le\infty$ and $I\subset\rre$, 
let us define a space-time Lebesgue spaces
\begin{eqnarray*}
L_t^qL_x^p(I)&=&
\{u:I\times\R\to\R\ ;\ \|u\|_{L_t^qL_x^p(I)}<\infty\},\\
\|u\|_{L_t^qL_x^p(I)}&=&
\|\|u(t,\cdot)\|_{L_x^p(\rre)}\|_{L_t^q(I)},\\
L_x^pL_t^q(I)&=&
\{u:I\times\R\to\R\ ;\ \|u\|_{L_x^pL_t^q(I)}<\infty\},\\
\|u\|_{L_x^pL_t^q(I)}&=&
\|\|u(\cdot,x)\|_{L_t^q(I)}\|_{L_x^p(\rre)}.
\end{eqnarray*}

The rest of the paper is organized as follows. 
In Section 2, we review the well-posedness theory 
for (\ref{gKdV}) in the Fourier-Lebesgue space. 
Sections 3 is devoted to the proof of Theorems \ref{T:lwp} and \ref{T:bu}.
In Sections 4 and 5, we prove 
Theorems \ref{thm1} and \ref{thm2},
respectively.

\section{Well-posedness in the Fourier-Lebesgue space}

In this section, we review the well-posedness theory 
for (\ref{gKdV}) in the Fourier-Lebesgue space $\hat{L}^{\alpha}$. 
Furthermore, we prove the long time perturbation 
for (\ref{gKdV}) in the Fourier-Lebesgue space.

We first review the space-time estimates in $\hat{L}^{\alpha}$ 
of solution to the Airy equation
\begin{eqnarray}
\left\{
\begin{array}{l}
\displaystyle{
\pt_tu+\pt_x^3u=F(t,x),
\qquad\  t\in I, x\in\rre,}\\
\displaystyle{u(0,x)=f(x),
\qquad\qquad\ \ x\in\rre,}
\end{array}
\right.
\label{eq:A}
\end{eqnarray}
where $I\subset\rre$ is an interval, 
$F:I\times\rre\to\rre$ and $f:\rre\to\rre$ are 
given functions. 
Let $\{V(t)\}_{t\in\rre}$ be an unitary 
group in $L^2$ defined by 
\begin{eqnarray*}
(V(t)f)(x)
=\frac{1}{\sqrt{2\pi}}
\int_{-\infty}^{\infty}e^{ix\xi+it\xi^{3}}
\hat{f}(\xi)d\xi.
\end{eqnarray*}
Using the group, 
the solution to (\ref{eq:A}) can be written as 
\begin{eqnarray*}
u(t)=V(t)f+\int_{0}^{t}V(t-\tau)F(\tau)d\tau.
\end{eqnarray*}


\begin{proposition}[homogeneous space-time estimates] 
\label{ho} 
Let $I$ be an interval. 
Let $(p,q)$ satisfy
\[
	0\le\frac{1}{p}<\frac14,\quad\quad 0 \le \frac1q 
	< \frac12 - \frac1p.
\]
Then, for any $f\in\hat{L}^r$,
\begin{equation}\label{eq:mixed}
	\norm{|D_x|^s V(t) f}_{L^p_x L^q_t(I)} 
	\le C\norm{f}_{\hat{L}^r},
\end{equation}
where
\[
	\frac1r = \frac2p + \frac1q,\quad s=-\frac1p+\frac2q
\]
and positive constant $C$ depends 
only on $r$ and $s$. 
\end{proposition}

\begin{proof}[Proof of Proposition \ref{ho}] 

For the proof of (\ref{eq:mixed}) with $(p,q,r)=(4,\infty,2)$ 
or $(p,q,r)=(\infty,2,2)$, see \cite[Theorem 2.5]{KPV2} and 
\cite[Theorem 4.1]{KPV2}, respectively. 
For the proof of (\ref{eq:mixed}) with $p=q$ and $r>4/3$, see 
Gr\"{u}nrock \cite[Corollary 3.6]{G} or \cite[Lemma 2.2]{MS}. 
The general case follows from the above cases and 
the interpolation. See \cite[Proposition 2.1]{MS} for the detail. 
\end{proof}

%

\begin{proposition}[inhomogeneous space-time estimates] 
\label{inho} 
Let $4/3<r<4$ and let 
$(p_{j},q_{j})$ ($j=1,2$) satisfy
\[
	0\le\frac{1}{p_{j}}<\frac14,\quad\quad 0 \le \frac{1}{q_{j}} 
	< \frac12 - \frac1p_{j}.
\]
Then, the inequalities 
\begin{equation}
\left\|
\int_0^tV(t-\tau)F(\tau)d\tau
\right\|_{L_t^{\infty}(I;\hat{L}_x^{r})}
\le C_{1}
\||D_x|^{-s_{2}}
F\|_{L_x^{p_{2}'}L_{t}^{q_{2}'}(I)},
\label{k}
\end{equation}
and
\begin{equation}
\left\||D_x|^{s_{1}}
\int_0^tV(t-\tau)F(\tau)d\tau
\right\|_{L_x^{p_{1}}L_{t}^{q_{1}}(I)}
\le C_{2}
\||D_x|^{-s_{2}}
F\|_{L_x^{p_{2}'}L_{t}^{q_{2}'}(I)} \label{l}
\end{equation}
hold for any $F$ satisfying $|D_x|^{-s_{2}}
F\in L_x^{p_{2}'}L_{t}^{q_{2}'}$ with
\[
	\frac1r = \frac{2}{p_{1}} + \frac{1}{q_{1}},
	\quad s_{1}=-\frac{1}{p_{1}}+\frac{2}{q_{1}}
\]
and
\[
	\frac{1}{r'} = \frac{2}{p_{2}} + \frac{1}{q_{2}},
	\quad s_{2}=-\frac{1}{p_{2}}+\frac{2}{q_{2}},
\]
where 
the constant $C_{1}$ depends on $r$, $s_{2}$ and $I$, and 
the constant $C_{2}$ depends on $r$, $s_{1}$, $s_{2}$ and $I$.
\end{proposition}

\begin{proof}[Proof of Proposition \ref{inho}] 
(\ref{k}) and (\ref{l}) follow from Proposition \ref{ho} and 
Christ-Kiselev lemma \cite{CK} (see also \cite[Lemma 2.5]{KK} 
for the space-time norm version of Christ-Kiselev lemma). 
See \cite[Proposition 2.5]{MS} for the detail. 
\end{proof}

Next, we review the small data scattering in $\hat{L}^{\alpha}$ 
for (\ref{gKdV}) obtained by \cite{MS}.  
%

\begin{lemma} \label{lem:ss}
Let $8/5<\alpha<2$. Then there exists $\tilde{\varepsilon}>0$ 
such that if $u_{0}\in \hat{L}_{x}^{\alpha}(\rre)$ 
satisfies $\|u_{0}\|_{\hat{L}_{x}^{\alpha}}\le\tilde{\varepsilon}$, 
then there exists a global $\hat{L}^{\alpha}$-solution $u$ to (\ref{gKdV}) 
satisfying 
\begin{eqnarray}
\|u\|_{L_{t}^{\infty}(\R; \hat{L}_{x}^{\alpha})}
+\|u\|_{S(\R)}
\le2\|u_{0}\|_{\hat{L}_{x}^{\alpha}}.
\label{se}
\end{eqnarray} 
Further, $u(t)$ scatters in $\hat{L}^\alpha$
for both time directions.
\end{lemma}

\begin{proof}[Proof of Lemma \ref{lem:ss}] 
See \cite[Theorem 1.7]{MS}. 
\end{proof}

Next we prove the long time perturbation lemma 
for (\ref{gKdV}) in the Fourier-Lebesgue space. 
We define
\begin{eqnarray*}
X(I)&:=&\{u:I\times\R\to\R\ ;\ \| u \|_{X(I)}<\infty\},\\
\| u \|_{X(I)}& := &
\| |D_x|^s u 
\|_{L^{\frac{20\alpha}{10-3\alpha}}_xL_t^{\frac{10}3} ( I)}, \\
Y(I)&:=&\{u:I\times\R\to\R\ ;\ \| u \|_{Y(I)}<\infty\},\\
\| u \|_{Y(I)}& := &\| |D_x|^s u \|_{L^{\frac{20\alpha}{10+13\alpha}}_xL_t^{\frac{10}7} ( I)}
\end{eqnarray*}
with $s=\frac34 - \frac{1}{2\alpha}$.

\begin{proposition}[Long time perturbation]\label{prop:long}
Assume $8/5<\alpha<2$. 
For any $M>0$ there exists $\eps>0$
such that the following property holds:
Let $t_0 \in \R$ and
let $I \subset \R$ be an interval such that $t_0 \in \overline{I}$.
Let $\widetilde{u}: I \times \R \to \R$ be a function such that $\widetilde{u} \in S(I) \cap X(I)$, 
where $S(I)$ is given by (\ref{sn}). 
Put $\mathcal{E} := 
(\pt_t+\pt_x^3)\widetilde{u} 
- \mu \partial_x (|\widetilde{u}|^{2\alpha}\widetilde{u})$.
Let $u_0 \in \hat{L}^\alpha$.
Suppose that
\begin{eqnarray}
\|\widetilde{u}\|_{S(I) \cap X(I)}  \le M
\label{ba1}
\end{eqnarray}
and
\begin{align}
\left\|V(t-t_0)(u_0-\widetilde{u}(t_0))
-\int_{t_0}^t V(t-\tau) \mathcal{E}(\tau) d\tau\right\|_{S (I) \cap X(I)}\le \eps.
\label{ba2}
\end{align}
Then, the unique $\hat{L}^\alpha$-solution $u(t)$ of \eqref{gKdV} satisfying $u(t_0) = u_0$
exists on $I$ and satisfies
\[
		\| u- \widetilde{u} \|_{S(I) \cap X(I)} \lesssim_M \eps.
\]
\end{proposition}

To prove Proposition \ref{prop:long}, we use the 
Leibniz rule for the fractional derivatives obtained by 
\cite{CW} and \cite{KPV}.

\begin{lemma}\label{2.5} 
Assume $\beta \in (0,1)$. Let $p, p_1, p_2, q, q_2
\in (1, \infty)$ and $q_1 \in (1,\infty] 
$ satisfy $1/p=1/p_1 + 1/p_2$ and $1/q=1/q_{1}+1/q_{2}$. 
We also assume $F\in C^{1}(\rre;\rre)$. Then for any interval $I$, 
the inequality
\begin{equation}
\| { |D_x|^{\beta} F(f)}\|_{L_x^pL_t^q(I)}
\lesssim \| {F'(f)}\|_{L_x^{p_1}L_t^{q_1}(I)} \| {|D_x|^{\beta}f }
\|_{L_x^{p_2}L_t^{q_2}(I)}
\label{Le}
\end{equation}
holds for any $f$ satisfying $F'(f)\in L_x^{p_1}L_t^{q_1}(I)$ and 
$|D_x|^{\beta}f\in L_x^{p_2}L_t^{q_2}(I)$, 
where the implicit constant depends only on $\beta, p_{1}, p_{2}, q_{1}, q_{2}$ 
and $I$.
\end{lemma}

\begin{proof}[Proof of Lemma \ref{2.5}]
See \cite[Proposition 3.1]{CW} and \cite[Theorem A.6]{KPV}. 
Note that the alternative proof of the inequality (\ref{Le})
can be found in \cite[Lemma 3.7]{MS}. 
\end{proof}

\begin{lemma}\label{2.4} 
Let $\beta \in (0,1), \beta_1, \beta_2 \in [0,\beta]$ 
satisfy $\beta= \beta_1 + \beta_2$ and let  
$p, p_1, p_2, q, q_1, q_2\in (1, \infty) $ satisfy $1/p=1/p_1 + 1/p_2$ 
and $1/q=1/q_1 + 1/q_2$. 
Then for all interval $I$, 
the inequality
\begin{eqnarray*}
\lefteqn{\| { |D_x|^{\beta}(fg) - f |D_x|^{\beta} g - g |D_x|^{\beta} f }\|_{L_x^pL_t^q(I)}}
\qquad\qquad\\
&\le& C \| {|D_x|^{\beta_1}f}\|_{L_x^{p_1}L_t^{q_1}(I)} 
\| {|D_x|^{\beta_2}g }\|_{L_x^{p_2}L_t^{q_2}(I)}
\end{eqnarray*}
holds for any $f$ and $g$ satisfying 
$|D_x|^{\beta_1}f\in L_x^{p_1}L_t^{q_1}(I)$ 
and $|D_x|^{\beta_2}g\in L_x^{p_2}L_t^{q_2}(I)$, 
where the implicit constant 
depends only on $\beta_{1}, \beta_{2}, p_{1}, p_{2}, q_{1}, q_{2}$ 
and $I$.
\end{lemma}

\begin{proof}[Proof of Lemma \ref{2.4}] 
See \cite[Theorem A.8]{KPV}. 
\end{proof}

\begin{proof}[Proof of Proposition \ref{prop:long}]
It suffices to consider the case $\inf I=t_0$. 
The general case follows by 
splitting $I=(I \cap [t_0,\infty)) \cup (I \cap (-\infty,t_0])$ 
and applying the time reversal symmetry to estimate the latter.
Further, we may let $t_0=0$ without loss of generality
by the time translation symmetry.

By the assumption (\ref{ba1}), we see that 
for any $\eta>0$ there exist $N=N(M,\eta)$ 
and a subdivision $\{t_j\}_{j=0}^N$
of $[0,\infty)$ with 
$0=t_0<t_1<\cdots<t_N=+\infty$ such that 
\[
	\|\widetilde{u}\|_{S(I_j)}+\|\widetilde{u}\|_{X(I_j)}<\eta
\]
holds for all $i \in [1,N]$,
where $I_j:=[t_{j-1},t_{j})$.

Let us first consider the equation for $w:=u-\widetilde{u}$ on $I_1=[0,t_1)$:
\begin{eqnarray}
	w(t)  &=& \mu \int_0^t V(t-\tau)\partial_x (|w+\widetilde{u}|^{2\alpha}(w+\widetilde{u})- |\widetilde{u}|^{2\alpha} \widetilde{u} ) d\tau \label{w}\\
	& &+N(t),\nonumber
\end{eqnarray}
where
\[
	N(t):=V(t)(u_0-\widetilde{u}(0))
	- \int_0^t V(t-\tau)\mathcal{E}(\tau)  d\tau.
\]
By Proposition \ref{inho} and Lemma \ref{2.5}, 
we obtain
\begin{align*}
	\|w\|_{S(I_1)\cap X(I_1)} \le{}&
	\|N\|_{S(I_1) \cap X(I_1)} \\
	&+ C (\|w\|_{X(I_1)} +\|\widetilde{u}\|_{X(I_1)})
	(\|w\|_{S(I_1)}^{2\alpha-1} +\|\widetilde{u}\|_{S(I_1)}^{2\alpha-1}) \|w\|_{S(I_1)} \\
	&+ C (\|w\|_{S(I_1)}^{2\alpha} +\|\widetilde{u}\|_{S(I_1)}^{2\alpha}) \|w\|_{X(I_1)} \\
	\le{}& \eps + C (\|w\|_{X(I_1)} +\eta)(\|w\|_{S(I_1)}^{2\alpha-1} +\eta^{2\alpha-1}) \|w\|_{S(I_1)} \\
	&+ C (\|w\|_{S(I_1)}^{2\alpha} +\eta^{2\alpha}) \|w\|_{X(I_1)} \\
	\le{}& \eps + C \eta^{2\alpha} \|w\|_{S(I_1)\cap X(I_1)} + C \|w\|_{S(I_1)\cap X(I_1)}^{2\alpha+1}.
\end{align*}
We remark that $C$ can be chosen independently of $M$, $\eta$, and $\eps$.
If $\eta$ is small then
this implies
\[
	\|w\|_{S(I_1)\cap X(I_1)} \le 2\eps + 2C \|w\|_{S(I_1)\cap X(I_1)}^{2\alpha+1}.
\]
There exists a constant $\delta>0$ such that 
if $2\eps\le \delta$ then this implies that
\[
	\|w\|_{S(I_1)\cap X(I_1)} \le 4\eps.
\]

Now, let $j\in[2,N]$ and
suppose that we can choose $\eps_{j-1}$ so that if $\eps \le \eps_{j-1}$ then
\[
	\|w\|_{S(I_k)\cap X(I_k)} \le 4^k \eps \le \eta
\]
holds for $k\in[1,j-1]$.
Let us next consider the equation (\ref{w}) 
for $w$ on $I_j=[t_{j-1},t_j)$. We rewrite (\ref{w}) as 
\begin{align*}
	w(t)
	={}&  \mu \int_0^{t_{j-1}} V(t-\tau)\partial_x (|w+\widetilde{u}|^{2\alpha}(w+\widetilde{u})- |\widetilde{u}|^{2\alpha} \widetilde{u} ) d\tau \\
	&+\mu \int_{t_{j-1}}^t V(t-\tau) \pt_x (|w+\widetilde{u}|^{2\alpha}(w+\widetilde{u})- |\widetilde{u}|^{2\alpha} \widetilde{u} ) d\tau 
	+N(t).
\end{align*}
By Proposition \ref{inho}, one has
\begin{align*}
	&\left\| \int_0^{t_{j-1}} V(t-\tau)\partial_x 
	(|w+\widetilde{u}|^{2\alpha}(w+\widetilde{u})
	- |\widetilde{u}|^{2\alpha} \widetilde{u} ) 
	d\tau\right\|_{S(I_j)\cap X(I_j)} \\
	&= \left\| \int_0^{t} V(t-\tau){\bf 1}_{[0,t_{j-1})}(\tau) 
	\partial_x (|w+\widetilde{u}|^{2\alpha}(w+\widetilde{u})
	- |\widetilde{u}|^{2\alpha} \widetilde{u} ) 
	d\tau\right\|_{S(I_j)\cap X(I_j)} \\
	&\le \left\| \int_0^{t} V(t-\tau){\bf 1}_{[0,t_{j-1})}(\tau) 
	\partial_x (|w+\widetilde{u}|^{2\alpha}(w+\widetilde{u})
	- |\widetilde{u}|^{2\alpha} \widetilde{u} ) 
	d\tau\right\|_{S([0,t_j))\cap X([0,t_j))} \\
	&\lesssim \|{\bf 1}_{[0,t_{j-1})} 
	\partial_x (|w+\widetilde{u}|^{2\alpha}(w+\widetilde{u})
	- |\widetilde{u}|^{2\alpha} \widetilde{u} )\|_{Y([0,t_j))} \\
	&= \| \partial_x (|w+\widetilde{u}|^{2\alpha}(w+\widetilde{u})
	- |\widetilde{u}|^{2\alpha} \widetilde{u} )\|_{Y([0,t_{j-1}))}\\
	& \le \sum_{k=1}^{j-1} \| 
	\partial_x (|w+\widetilde{u}|^{2\alpha}(w+\widetilde{u})
	- |\widetilde{u}|^{2\alpha} \widetilde{u} )\|_{Y(I_k)}\\
	& \le \sum_{k=1}^{j-1} 2C\eta^{2\alpha} 4^k\eps 
	\le \frac83 C \eta^{2\alpha} 4^{j-1} \eps.
\end{align*}
Hence, 
\begin{align*}
	\|w\|_{S(I_j)\cap X(I_j)} \le{}&
	\|N\|_{S(I_j) \cap X(I_j)} 
	+ \frac83 C \eta^{2\alpha} 4^{j-1} \eps\\
	&+ C (\|w\|_{X(I_j)} 
	+\|\widetilde{u}\|_{X(I_j)})
	(\|w\|_{S(I_j)}^{2\alpha-1} 
	+\|\widetilde{u}\|_{S(I_j)}^{2\alpha-1}) \|w\|_{S(I_j)} \\
	&+ C (\|w\|_{S(I_j)}^{2\alpha} 
	+\|\widetilde{u}\|_{S(I_j)}^{2\alpha}) \|w\|_{X(I_j)} \\
	\le{}& \eps +\frac83 C \eta^{2\alpha} 4^{j-1} \eps
	+ C \eta^{2\alpha} \|w\|_{S(I_j)\cap X(I_j)} 
	+ C \|w\|_{S(I_j)\cap X(I_j)}^{2\alpha+1}.
\end{align*}
Letting $\eta$ even smaller if necessary, we have $C\eta^{2\alpha} \le \frac14$ and hence
\[
	\|w\|_{S(I_j)\cap X(I_j)} \le
	\tfrac43(1 +\tfrac23 4^{j-1}) \eps + 2C \|w\|_{S(I_j)\cap X(I_j)}^{2\alpha+1}.
\]
Hence, if $\eps \le \min ( \frac23(1 +\tfrac23 4^{j-1})^{-1} \delta, 4^{-j}\eta,\eps_{j-1}) =:\eps_j$ then
\[
	\|w\|_{S(I_j)\cap X(I_j)} \le \tfrac83(1 +\tfrac23 4^{j-1}) \eps \le 4^j \eps \le \eta.
\]
Hence, by induction, we can choose $\eps_N$ such that if 
$\eps \le \eps_N$ then
\[
	\|w\|_{S(I_j)\cap X(I_j)} \le 4^j \eps \le \eta
\]
holds for $j\in[1,N]$.
Combining this estimate and noting that $N$ depends on $M$, we obtain
\[
	\|w\|_{S(I)\cap X(I)} \lesssim_{M} \eps.
\]
%
\end{proof}

In the end of this section, we prove the compactness of the embedding 
$H^1 \cap H^{0,1} \hookrightarrow \hat{L}^\alpha$.

\begin{lemma}\label{L:tb}
The embedding $H^1 \cap H^{0,1} \hookrightarrow \hat{L}^\alpha$ 
is compact for $1\le \alpha \le \infty$.
\end{lemma}

\begin{proof}[Proof of Lemma \ref{L:tb}]
It is an immediate consequence of the embedding 
$H^{3/4} \cap H^{0,3/4} \hookrightarrow L^1 \cap L^\infty$ 
holds and 
the fact that the embedding $H^1 \cap H^{0,1} 
\hookrightarrow H^{3/4} \cap H^{0,3/4}$ is compact.
\end{proof}


\section{Proof of Theorems \ref{T:lwp} and \ref{T:bu}} 

In this section, we prove local well-posedenss and blowup alternative.
Fix $t_0$ and $u_0 \in H^1$ such that $J(t_0)u_0 \in L^2$.
Note that $u_0 \in H^1 \cap \hat{L}^\alpha$.
Indeed,
\begin{equation}\label{E:hLaJ}
\begin{aligned}
	\|\mathcal{F} u_0\|_{L^{\alpha'}}
	&=\| \mathcal{F} V(-t_0)u_0 \|_{L^{\alpha'}}\\
	&{}\lesssim 
	\| \mathcal{F} V(-t_0)u_0 \|_{L^2}^{\frac{3\alpha-2}{2\alpha}}
	\| \partial_\xi \mathcal{F} V(-t_0)u_0 \|_{L^2}^{\frac{2-\alpha}{2\alpha}}\\
	&{}=\|u_0\|_{L^2}^{\frac{3\alpha-2}{2\alpha}}
	\|J(t_0)u_0\|_{L^2}^{\frac{2-\alpha}{2\alpha}}<\infty.
\end{aligned}
\end{equation}
Hence, by the local well-posedness result in $\hat{L}^\alpha \cap H^1$ \cite[Theorem 1.5]{MS}, one obtains
a $\hat{L}^\alpha \cap H^1$-solution $u$ to \eqref{gKdV} in a neighborhood $I$ of $t_0$. In particular, one has
\[
	\| u \|_{L^\infty_t(I;H^1_x(\R))}+\sum_{k=1}^2\| \partial_x^k u \|_{L^\infty_x(\R;L^2_t(I))}
	\lesssim \|u_0\|_{H^1}. 
\]
We note that the size of the neighborhood is chosen so that $\|V(t-t_0)u_0\|_{S(I)} $ is smaller than a universal constant. 
Hence, what we have to do is to show that the $H^1 \cap \hat{L}^\alpha$-solution $u$ is a $H^1 \cap H^{0,1}$-solution.
To this end, we estimate $Ju$ by considering 
\begin{eqnarray*}
v:=Ju+3\mu t|u|^{2\alpha}u 
\end{eqnarray*}
defined in \eqref{v}.
We further introduce
\begin{eqnarray*}
P:=x\pt_x+3t\pt_t.
\end{eqnarray*}
We have the identity
\begin{equation}\label{E:Puv}
	\partial_x v = Pu + u.
\end{equation}

Before the proof, let us 
derive an equation for $Ju$ and $Pu$.
We also confirm that $v$ solves \eqref{E:v}.
Let $L=\pt_t+\pt_x^3$. 
Suppose that $u \in C(I;H^1)$ solves
\[
	Lu = \mu \partial_x (|u|^{2\alpha} u)
\]
in the distribution sense.
Let us note beforehand that the following calculation is valid in the distribution sense.
Operating $J$ to the both sides 
 and noting $[L,J]=0$, we see
\begin{eqnarray*}
LJu&=&\mu J\pt_x(|u|^{2\alpha}u).
\end{eqnarray*}
It holds that
\begin{eqnarray}
J\pt_x=P-3tL.
\end{eqnarray}
Hence, we have
\begin{eqnarray}
LJu&=&\mu P(|u|^{2\alpha}u)-3\mu tL(|u|^{2\alpha}u)
\label{L1}
\\
&=&(2\alpha+1)\mu|u|^{2\alpha}Pu-3\mu tL(|u|^{2\alpha}u).\nonumber
\end{eqnarray}
Since $[J,\pt_x]=-1$, another use of the above identity yields
\begin{eqnarray}
Pu&=&J\pt_xu+3tLu=\pt_xJu-u+3\mu t\pt_x(|u|^{2\alpha}u)
\label{L2}\\
&=&\pt_xv-u.\nonumber
\end{eqnarray}
This is \eqref{E:Puv}.
Furthermore, since $[L,t]=1$, we have
\begin{eqnarray}
tL(|u|^{2\alpha}u)=Lt(|u|^{2\alpha}u)-|u|^{2\alpha}u.
\label{L3}
\end{eqnarray}
Substituting (\ref{L2}) and (\ref{L3}) into (\ref{L1}), 
we obtain
\begin{eqnarray*}
Lv=(2\alpha+1)\mu|u|^{2\alpha}\pt_xv
-2(\alpha-1)\mu|u|^{2\alpha}u,
\end{eqnarray*}
which is nothing but \eqref{E:v}.
On the other hand, if we operate $P$ to the  equation for $u$, 
we obtain
\[
	PLu = \mu P \pt_x (|u|^{2\alpha}u).
\]
Using the relations $[P,L]=-3L$ and $[P,\pt_x]=-\pt_x$, 
we obtain
\begin{align}
	LPu ={}& 3L u + \mu \pt_x P (|u|^{2\alpha}u)
	- \mu \pt_x (|u|^{2\alpha}u)\label{L4} \\
	={}& (2\alpha+1)\mu \pt_x  (|u|^{2\alpha}Pu)
	+ 2\mu \pt_x (|u|^{2\alpha}u).\nonumber
\end{align}

Thus, we see from \eqref{L1} that
\begin{equation}\label{E:Ju}
	\partial_t (Ju) + \partial_x^3 (Ju) = (2\alpha+1)\mu|u|^{2\alpha}Pu-3\mu t(\pt_t + \pt_x^3)(|u|^{2\alpha}u).
\end{equation}
Further, by (\ref{L4}),
\begin{equation}\label{E:Pu}
	\partial_t (Pu) + \partial_x^3 (Pu) = (2\alpha+1)\mu \pt_x(|u|^{2\alpha}Pu)+2\mu \pt_x(|u|^{2\alpha}u).
\end{equation}

The local well-posedness in the weighted Sobolev space $H^1 \cap H^{0,1}$ (Theorem \ref{T:lwp}) is a consequence of 
the following persistence-type result.
\begin{lemma}\label{L:localv}
Let $t_0 \in \R$ and let $u_0 \in \hat{L}^\alpha \cap H^1$.
Let $u(t)$ be a $\hat{L}^\alpha \cap H^1$-solution to \eqref{gKdV} under \eqref{gIC}.
There exists a constant $\delta>0$ such that
if $V(-t_0)u_0 \in H^1 \cap H^{0,1}$ then $u(t)$ is a
$H^1 \cap H^{0,1}$-solution to \eqref{gKdV} on any
interval $I\ni t_0$ satisfying $\|u\|_{S(I)} \le \delta$.
Further,  
\begin{eqnarray*}
	\lefteqn{\|Ju\|_{L^\infty_tL^2_x(I)}+\|v\|_{L^\infty_tL^2_x(I)} 
	+ \|\pt_x v \|_{L^\infty_x L^2_t(I)}}\qquad\qquad\qquad \\
	&\lesssim& 
	\|V(-t_0)u_0\|_{H^1_x} 
	+ \Jbr{t_0}\|u_0\|_{H^1_x}^{2\alpha+1},
\end{eqnarray*}
where $v$ is defined by \eqref{v}.
\end{lemma}

\begin{proof}[Proof of Lemma \ref{L:localv}.]  

Let us prove that the $H^1 \cap \hat{L}^\alpha$-solution 
satisfies the desired weighted estimate.
To this end, we obtain an estimate of $v$ defined in \eqref{v} 
by solving \eqref{E:v} under the initial condition
\[
	v(t_0) = v_0:=J(t_0)u_0 + 3\mu t_0 |u_0|^{2\alpha} u_0 \in L^2.
\]
For 
$R>0$ and $T>0$, we define a complete metric space
\[
	Z_{R,T}:=\{ v \in C(I_T; L^2_x)\ ; \ \| v \|_{Z(I_T)}
  \le R	\}
\]
with the distance 
\[
	d(v_1,v_2)=\| v_1-v_2\|_{Z(I_T)} ,
\]
where $I_T=(t_0-T,t_0+T)$,
\begin{equation}\label{E:ZT}
	\| v \|_{Z(I)} := \| v \|_{L^\infty_tL^2_x(I)}
	+\| \partial_x v \|_{L^\infty_xL^2_t(I)}.
\end{equation}
We suppose that $T>0$ is small so that $I_T \subset I$.
Let us prove that the map $\Phi(v) $ 
 defined by
\begin{align*}
	\Phi(v)(t) :={}& 
	V(t-t_0)v_0 +(2\alpha+1) \mu 
	\int_{t_0}^t V(t-\tau) (|u|^{2\alpha}\partial_xv)(\tau)d\tau \\
	&-2(\alpha-1)\mu \int_{t_0}^t V(t-\tau)(|u|^{2\alpha}u)(\tau)d\tau
\end{align*}
is a contraction map from $Z_{R,T}$ to itself.
 
Pick $v \in Z_{R,T}$.
By Propositions \ref{ho} and \ref{inho},
%
\begin{align*}
	\| \Phi(v) \|_{Z(I_T)} 
	&\le C \|v_0\|_{L^2_x} + C \||u|^{2\alpha} \partial_x v \|_{L^{\frac54}_xL^{\frac{10}9}_t(I_T)} + C \||u|^{2\alpha}u\|_{L^1_tL^2_x(I_T)}	\\
	&\le C \|v_0\|_{L^2_x} + C\|u\|_{S(I_T)}^{2\alpha} \| \partial_x v \|_{L^\infty_xL^2_t(I_T)} + CT \|u\|_{L^\infty_t H^1_x(I_T)}^{2\alpha+1}\\
	&\le C \|v_0\|_{L^2_x} + C\|u\|_{S(I_T)}^{2\alpha}R + CT \|u_0\|_{H^1_x}^{2\alpha+1}.
\end{align*}
We first choose $T\le 1$ so small that 
$C\|u\|_{S(I_T)}^{2\alpha} \le \frac12$
and then we let
\[
	R=2C( \|v_0\|_{L^2_x}+  \|u_0\|_{ H^1_x}^{2\alpha+1}). \]
 Then, one sees that $\Phi$ is a map from $Z_{R,T}$ to itself.
Similarly, for $v_1,v_2 \in Z_{R,T}$, one obtains
\[
	\Phi(v_1) - \Phi(v_2) = 
	(2\alpha+1) \mu \int_{t_0}^t V(t-\tau)
	(|u|^{2\alpha}\partial_x(v_1-v_2))(\tau) d\tau
\]
and hence, estimating as above, one sees that
\[
	d(\Phi(v_1) , \Phi(v_2))
	\le  C\|u\|_{S(I_T)}^{2\alpha} \| \partial_x (v_1-v_2) \|_{L^\infty_xL^2_t(I_T)}
	\le \frac12 d(v_1,v_2),
\]
which shows that $\Phi$ is a contraction map.
Thus, we see that $v\in C(I_T;L^2_x)$ obeys the bound
\[
	\|v\|_{Z(I_T)} \le R \lesssim
	\|J(t_0)u_0\|_{L^2_x}+ \Jbr{t_0} \|u_0\|_{ H^1_x}^{2\alpha+1}.
\]

So far, we construct $v$ as a solution to \eqref{E:v}.
Let us prove that $v=Ju + 3\mu t |u|^{2\alpha} u$ holds 
in the distribution sense,
which implies that
$J(t)u(t) \in C(I_T;L^2_x)$ and
\[
	\|Ju\|_{L^\infty_t L^2_x(I_T)} \lesssim \|v\|_{Z(I_T)} 
	+\Jbr{t_0} \|u\|_{L^\infty(I_T;H^1)}^{2\alpha+1}
	\lesssim
	\|J(t_0)u_0\|_{L^2}+ \Jbr{t_0} \|u_0\|_{ H^1_x}	^{2\alpha+1}.
\]
To this end, we put
\[
	z = \partial_x v - u ,\qquad
	w = v - 3\mu t |u|^{2\alpha} u.
\]
By \eqref{E:v}, one obtains
\begin{align*}
	(\pt_t+\pt_x^3)z ={}& \pt_x ((2\alpha+1)\mu  |u|^{2\alpha} \partial_x v - 2(\alpha-1) \mu |u|^{2\alpha} u ) - \mu \pt_x (|u|^{2\alpha}u) \\
	={}& (2\alpha+1)\mu  \pt_x(|u|^{2\alpha} (z+u)) - (2\alpha-1)\mu \pt_x (|u|^{2\alpha}u) \\
	={}& (2\alpha+1)\mu  \pt_x(|u|^{2\alpha} z) +2 \mu \pt_x (|u|^{2\alpha}u).
\end{align*}
Hence, $z$ solves \eqref{E:Pu} in the distribution sense.
Together with
\[
	z(t_0) = \pt_x v(t_0) - u(t_0)
	= x \pt_x u_0 + 3 t_0 (-\pt_x^3 u_0 
	+ \mu \pt_x(|u_0|^{2\alpha}u_0)) 
	= (Pu)(t_0),
\]
we see that $z=Pu$.
Hence, we further obtain
\begin{align*}
	\pt_t w+ \pt_x^3 w = {}&
	(\pt_t + \pt_x^3 )v - 3\mu |u|^{2\alpha} u 
	- 3\mu t (\pt_t+\pt_x^3) (|u|^{2\alpha}u) \\
	={}& (2\alpha+1)\mu  |u|^{2\alpha} \partial_x v 
	- (2\alpha+1) \mu |u|^{2\alpha} u 
	- 3\mu t (\pt_t+\pt_x^3) (|u|^{2\alpha}u) \\
	={}& (2\alpha+1)\mu  |u|^{2\alpha} Pu 
	- 3\mu t (\pt_t+\pt_x^3) (|u|^{2\alpha}u) ,
\end{align*}
i.e., $w$ solves \eqref{E:Ju} in the distribution sense.
Since
\[
	w(t_0) = v(t_0) - 3\mu t_0 |u_0|^{2\alpha}u_0 = J(t_0)u_0,
\]
we see that $w=Ju$.
Thus, $v= Ju + 3\mu t |u|^{2\alpha} u$ holds.
\end{proof}

We conclude this section with the proof of Theorem \ref{T:bu}.

\begin{proof}[Proof of Theorem \ref{T:bu}.]
Let $u(t)$ be a maximal-lifespan  $H^1 \cap \hat{L}^\alpha$-solution given in \cite[Theorem 1.9]{MS}.
Here, the maximal lifespan $I_{\max} = (-T_{\min},T_{\max})$ is that as a $H^1 \cap \hat{L}^\alpha$-solution.
Let us prove that this is also a
maximal-lifespan in the sense of $H^1 \cap H^{0,1}$-solution.
Recall that
$T_{\max}<\infty$ implies 
\[
	\| u\|_{S([t_0,T_{\max}))}=\infty
\]
(see \cite[Theorem 1.5]{MS}).
Hence, it suffices to prove that, for any finite $T>t_0$,
\[
	\| u\|_{S([t_0,T))}<\infty \Longrightarrow
	\|Ju\|_{L^\infty_t L^2_x ([t_0,T))}<\infty.
\]

Let $\delta>0$ be the number given in Lemma \ref{L:localv}. 
We can obtain a subdivision $\{t_j\}_{j=1}^N$
of $[t_0,T)$:
\[
	t_0<t_1<t_2<\cdots<t_N=T
\]
so that $N\lesssim_{\alpha,\|u\|_{S([t_0,T))}}1$ and
$\| u\|_{S([t_{j-1},t_j))}\le \delta$ for all $j \in [1,N]$.
By applying Lemma \ref{L:localv} to each interval $[t_{j-1},t_j)$, we obtain $\|Ju\|_{L^\infty_t L^2_x ([t_{j-1},t_j))}<\infty$ for all $j\in [1,N]$.
This implies the desired boundedness
$	\|Ju\|_{L^\infty_t L^2_x([t_0,T))}<\infty$.
This completes the proof.
\end{proof}


\section{Proof of Theorem \ref{thm1}} 

To prove Theorem \ref{thm1}, we employ 
the well-posedness result of (\ref{gKdV}) 
in the Fourier-Lebesgue space $\hat{L}^{\alpha}(\rre)$ 
mentioned in Section 2.

\begin{lemma} \label{lem1}
Let $8/5<\alpha<2$. 
Let $t_0 \in \R$ and suppose that 
$V(-t_0)u_0 \in H^1 \cap H^{0,1}$.
There exists $\varepsilon_1>0$ such that if
$\varepsilon=\|V(-t_0)u_0\|_{H^1\cap H^{0,1}}\le\varepsilon_1$ 
then  the $H^1 \cap H^{0,1}$-solution to (\ref{gKdV}) 
under \eqref{gIC} is global and satisfies
\begin{eqnarray}
\|u\|_{S(\R)}
\lesssim\varepsilon.\label{ScatteringC}
\end{eqnarray} 
\end{lemma}

\begin{proof}[Proof of Lemma \ref{lem1}.]  
By \eqref{E:hLaJ}, we see 
that $\|u_0\|_{\hat{L}^\alpha}\lesssim \eps$. 
Hence Lemma \ref{lem:ss} yields that 
if $\eps$ is sufficiently small, then 
there exists a global $\hat{L}^\alpha$-solution $u$
satisfying (\ref{ScatteringC}).
By Theorem \ref{T:bu}, $u$ is 
a global $H^1\cap H^{0,1}$-solution.
%
\end{proof}

%
%

\begin{lemma}\label{lem2} 
Let $8/5<\alpha<2$. 
Let $t_0 \in \R$ and suppose that 
$V(-t_0)u_0 \in H^1 \cap H^{0,1}$.
Let $u$ be the unique maximal-lifespan 
$H^1 \cap H^{0,1}$-solution to \eqref{gKdV} 
under \eqref{gIC}.
Then there exists $\delta_2>0$ such that 
if an interval $I \ni t_0$ satisfies $I \subset I_{\max}$ and
\[
	\|u\|_{S(I)} \le \delta_2
\]
then it holds that
\begin{equation}
	\|u\|_{L^\infty_t H^1_x(I)} +  \|\pt_xu\|_{L_x^{\infty}L_t^2(I)}+\|\pt_x^2u\|_{L_x^{\infty}L_t^2(I)}\lesssim \|u_0\|_{H^1_x}.
\end{equation}
In particular,
 there exists $\varepsilon_2\in (0,\eps_1]$ such that if
$\varepsilon=\|V(-t_0)u_0\|_{H^1\cap H^{0,1}}\le\varepsilon_2$, then  the solution is global and satisfies \eqref{ScatteringC} and
\begin{eqnarray}
\|u\|_{L_t^{\infty}H_x^1(\R)}+
\|\pt_xu\|_{L_x^{\infty}L_t^2(\R)}+\|\pt_x^2u\|_{L_x^{\infty}L_t^2(\R)}\lesssim\varepsilon,
\label{ScatteringD}
\end{eqnarray}
where $\eps_1$ is the number given in Lemma \ref{lem1}.
\end{lemma}

\begin{proof}[Proof of Lemma \ref{lem2}]
The latter half follows from the former half and the previous lemma.
Hence, let us prove the former part.
We omit $(I)$ in the norm, for simplicity.

By Propositions \ref{ho} and \ref{inho}, we have
\begin{eqnarray}
\lefteqn{\|u\|_{L_t^{\infty}H_x^1}+
\|\pt_xu\|_{L_x^{\infty}L_t^2}
+\|\pt_x^2u\|_{L_x^{\infty}L_t^2}}\label{sharp}\\
&\lesssim&\|u_0\|_{H_x^1}
+\|\pt_x(|u|^{2\alpha}u)\|_{L_x^{\frac54}L_t^{\frac{10}{9}}}
+\|\pt_x^2(|u|^{2\alpha}u)\|_{L_x^{\frac54}L_t^{\frac{10}{9}}}
\nonumber\\
&\lesssim&\|u_0\|_{H_x^1}
+\|u\|_{S}^{2\alpha}
\|\pt_xu\|_{L_x^{\infty}L_t^2}
+\|u\|_{S}^{2\alpha-1}
\|\pt_xu\|_{L_x^{5\alpha}L_t^{\frac{20\alpha}{5\alpha+2}}}^2
\nonumber\\
& &+\|u\|_{S}^{2\alpha}
\|\pt_x^2u\|_{L_x^{\infty}L_t^2}.
\nonumber
\end{eqnarray}
Since
\begin{eqnarray*}
\|\pt_xu\|_{L_x^{5\alpha}L_t^{\frac{20\alpha}{5\alpha+2}}}
\lesssim\|u\|_{S}^{\frac12}
\|\pt_x^2u\|_{L_x^{\infty}L_t^2}^{\frac12},
\end{eqnarray*}
substituting this and Lemma \ref{lem1} into (\ref{sharp}), we 
obtain
\begin{eqnarray*}
\lefteqn{\|u\|_{L_t^{\infty}H_x^1}+
\|\pt_xu\|_{L_x^{\infty}L_t^2}+\|\pt_x^2u\|_{L_x^{\infty}L_t^2}}\\
&\lesssim&
\|u_0\|_{H_x^1}
+\|u\|_{S}^{2\alpha}
(\|\pt_xu\|_{L_x^{\infty}L_t^2}+\|\pt_x^2u\|_{L_x^{\infty}L_t^2})\\
&\lesssim&
\|u_0\|_{H_x^1}+\delta_2^{2\alpha}
(\|\pt_xu\|_{L_x^{\infty}L_t^2}+\|\pt_x^2u\|_{L_x^{\infty}L_t^2}).
\end{eqnarray*}
Hence if $\delta_2$ is sufficiently small, 
then we have the desired estimate. 
\end{proof}

\begin{corollary}\label{C:h1}
Let $8/5<\alpha<2$. 
Let $t_0 \in \R$ and suppose that $V(-t_0)u_0 \in H^1 \cap H^{0,1}$.
Let $u$ be the unique maximal-lifespan $H^1 \cap H^{0,1}$-solution to \eqref{gKdV} under \eqref{gIC}.
If $\|u\|_{S(I)}<\infty$ holds for an interval $I$ then we have
\begin{eqnarray*}
\|u\|_{L_t^{\infty}H_x^1(I)}+
\|\pt_xu\|_{L_x^{\infty}L_t^2(I)}+\|\pt_x^2u\|_{L_x^{\infty}L_t^2(I)}<\infty.
\end{eqnarray*}
\end{corollary}

\begin{proof}[Proof of Corollary \ref{C:h1}.]  
We subdivide the interval $I$ so that
$S$-norm of the solution on each subinterval is smaller than the constant $\delta_2$ in Lemma \ref{lem2}. 
Note that the number of the subinterval depends only on $\alpha$ and $\|u\|_{S(I)}$.
Then, a recursive use of Lemma \ref{lem2} yields the result.
\end{proof}

%
%
Now, let us turn to the global bound on $Ju$.

\begin{lemma}\label{lem3}
Let $8/5<\alpha<2$. 
Let $t_0 \in \R$ and suppose that $V(-t_0)u_0 \in H^1 \cap H^{0,1}$.
There exists $\varepsilon_3\in (0,\eps_2]$ such that if
$\varepsilon=\|V(-t_0)u_0\|_{H^1\cap H^{0,1}}\le\varepsilon_3$, then  
the unique global $H^1 \cap H^{0,1}$-solution to \eqref{gKdV} under \eqref{gIC} satisfies \eqref{ScatteringC}, \eqref{ScatteringD}, and
\begin{equation}
\sup_{t\in \R}\Jbr{t}^\frac13 \|u(t)\|_{L^\infty_x}+
\|Ju\|_{L_t^{\infty}L_x^2(\R)}+\|v\|_{L_t^{\infty}L_x^2(\R)} + \|\pt_xv\|_{L_x^{\infty}L_t^2(\R)}
\lesssim\varepsilon,\label{ScatteringE}
\end{equation}
where $\eps_2$ is the number given in Lemma \ref{lem2}.
\end{lemma}

To prove Lemma \ref{lem3}, we show the Klainerman-Sobolev type inequality.

\begin{lemma}[Klainerman-Sobolev type inequality]\label{ks} 
Let $t\neq 0$ and $p \in [2,\infty]$. 
For any $u\in L_x^2$ satisfying $J(t)u\in L_x^2$, 
we have
\begin{eqnarray*}
\|u\|_{L_x^{p}}
\lesssim |t|^{-\frac13+\frac{2}{3p}}\|u\|_{L_x^2}^{\frac12+\frac1p}\|Ju\|_{L_x^2}^{\frac12-\frac1p}.
\end{eqnarray*}
\end{lemma}

\begin{proof}[Proof of Lemma \ref{ks}]
We consider the case $p=\infty$. 
By the elementary property of 
the Airy function, 
we see
\begin{eqnarray*}
\|V(t)f\|_{L_x^{\infty}}\lesssim t^{-\frac13}\|f\|_{L_x^1}.
\end{eqnarray*}
Hence by the $L^2$ unitary property of the group $V(t)$, 
\begin{eqnarray*}
\|u(t)\|_{L_x^{\infty}}
&=&\|V(t)V(-t)u\|_{L_x^{\infty}}\\
&\lesssim&t^{-\frac13}\|V(-t)u\|_{L_x^1}\\
&\lesssim&t^{-\frac13}\|V(-t)u\|_{L_x^2}^{\frac12}
\|xV(-t)u\|_{L_x^2}^{\frac12}\\
&=&Ct^{-\frac13}\|u\|_{L_x^2}^{\frac12}
\|J(t)u\|_{L_x^2}^{\frac12}.
\end{eqnarray*}
Hence, we obtain the $L^\infty$-estimate.
Note that $L^2$-estimate is obvious by the unitary property of $V(t)$.
The general case follows by interpolation.
\end{proof}

\begin{proof}[Proof of Lemma \ref{lem3}] 
Suppose that $\eps \le\eps_2$. Then, the global solution $u(t)$ satisfies \eqref{ScatteringC} and \eqref{ScatteringD}.
We prove the bound \eqref{ScatteringE} on $[0,\infty)$.

By the definition of $v$, 
\begin{eqnarray*}
\|Ju\|_{L_x^2}\lesssim\|v\|_{L_x^2}+t\||u|^{2\alpha}u\|_{L_x^2}.
\end{eqnarray*}
By the Sobolev and the Kleinerman-Sobolev inequalities 
(Lemma \ref{ks}), 
\begin{eqnarray*}
\|u(t)\|_{L_x^{\infty}}
\lesssim\left\{
\begin{array}{l}
\displaystyle{
\|u\|_{H_x^1}\lesssim\varepsilon
\qquad\qquad\qquad\text{for}\ 0\le t\le1,}\\
\displaystyle{
t^{-\frac13}\|u\|_{L_x^2}^{\frac12}\|Ju\|_{L_x^2}^{\frac12}
\lesssim \varepsilon^{\frac12}t^{-\frac13}\|Ju\|_{L_x^2}^{\frac12}
\qquad\text{for}\ t\ge1.}
\end{array}
\right.
\end{eqnarray*}
Hence 
\begin{eqnarray}
\||u|^{2\alpha}u\|_{L_x^2}
\lesssim {\bf 1}_{[0,1]}(t)\varepsilon^{2\alpha+1}
+{\bf 1}_{[1,\infty]}(t)\varepsilon^{\alpha+1}t^{-\frac23\alpha}
\|Ju\|_{L_x^2}^{\alpha},
\label{N1}
\end{eqnarray}
where ${\bf 1}_{A}$ is a characteristic function on the set $A$. 
Therefore, for any $T>1$
\begin{eqnarray}
\|Ju\|_{L_t^{\infty}L_x^2(I_T)}\lesssim\|v\|_{L_t^{\infty}L_x^2(I_T)}+\varepsilon^{2\alpha+1}
+\varepsilon^{\alpha+1}
\|Ju\|_{L_t^{\infty}L_x^2(I_T)}^{\alpha},\label{N2}
\end{eqnarray}
where $I_T=[0,T)$.
By Propositions \ref{ho} and 
\ref{inho}, \eqref{ScatteringC}, and (\ref{N1}), 
\begin{eqnarray}
\lefteqn{\|v\|_{L_t^{\infty}L_x^2(I_T)}
+\|\pt_xv\|_{L_x^{\infty}L_t^2(I_T)}}\label{N3}\\
&\lesssim&
\|xu_0\|_{L_x^2}+\||u|^{2\alpha}\pt_xv
\|_{L_x^{\frac54}L_t^{\frac{10}{9}}(I_T)}
+\||u|^{2\alpha}u\|_{L_t^1L_x^2(I_T)}\nonumber\\
&\lesssim&
\|xu_0\|_{L_x^2}+\|u\|_{S(I_T)}^{2\alpha}
\|\pt_xv\|_{L_x^{\infty}L_t^{2}(I_T)}+\||u|^{2\alpha}u\|_{L_t^1L_x^2(I_T)}\nonumber\\
&\lesssim&
\varepsilon+\varepsilon^{2\alpha}\|\pt_xv\|_{L_x^{\infty}L_t^{2}(I_T)}+\varepsilon^{2\alpha+1}
+\varepsilon^{\alpha+1}\|Ju\|_{L_t^{\infty}L_x^2(I_T)}^{\alpha}.\nonumber
\end{eqnarray}
By (\ref{N2}) and (\ref{N3}), 
\begin{align*}
&\|Ju\|_{L_t^{\infty}L_x^2(I_T)}+\|v\|_{L_t^{\infty}L_x^2(I_T)}+\|\pt_xv\|_{L_x^{\infty}L_t^2(I_T)} \\
&\lesssim\varepsilon+\varepsilon^{2\alpha}\|\pt_xv\|_{L_x^{\infty}L_t^{2}(I_T)}+\varepsilon^{2\alpha+1}
+\varepsilon^{\alpha+1}\|Ju\|_{L_t^{\infty}L_x^2(I_T)}^{\alpha}.
\end{align*}
Hence letting $\|u\|_{A_T}:=\|Ju\|_{L_t^{\infty}L_x^2(I_T)}+\|v\|_{L_t^{\infty}L_x^2(I_T)}+\|\pt_xv\|_{L_x^{\infty}L_t^2(I_T)}$, 
we have
\begin{eqnarray*}
\|u\|_{A_T}
\lesssim\varepsilon+\varepsilon^{2\alpha}\|u\|_{A_T}+\|u\|_{A_T}^{\alpha}.
\end{eqnarray*}
Hence if $\varepsilon$ is sufficiently small, 
then this inequality implies that $\|u\|_{A_T} \lesssim \eps$.
Since $T>1$ is arbitrary, we have
$\|u\|_{A_\infty} \lesssim \eps$.
Finally, combining 
$\|u\|_{A_\infty} \lesssim \eps$ and Lemma \ref{ks}, 
we have
\[
	\sup_{t\in \R} \Jbr{t}^{\frac13} \|u(t)\|_{L^\infty} \lesssim \eps.
\]
This completes the proof of
 (\ref{ScatteringE}). 
\end{proof}

We now turn to the scattering in $H^1 \cap H^{0,1}$.

\begin{lemma}\label{L:scatter1}
Let $8/5<\alpha<2$. 
Let $u$ be a maximal-lifespan $H^1 \cap H^{0,1}$-solution to \eqref{gKdV}.
Pick $t_0 \in I_{\max}$.
If
\[
	\|u\|_{S([t_0,T_{\max}))} + \|Ju\|_{L^\infty_t L^2_x ([t_0,T_{\max}))} <\infty,
\]
then $u(t)$ scatters in $H^1 \cap H^{0,1}$ for positive time direction.
A similar statement holds for the negative time direction.
\end{lemma}

\begin{proof}[Proof of Lemma \ref{L:scatter1}] 
By the blowup criteria, we have $T_{\max}=\infty$.
Hence, without loss of generality, we may suppose that $t_0>0$.
By Corollary \ref{C:h1}, we have
\[
	\|u\|_{L^\infty_t H^1_x([t_0,\infty))} + \|\pt_x u\|_{L^\infty_x  L^2_t([t_0,\infty))} + \|\pt_x^2 u\|_{L^\infty_x  L^2_t([t_0,\infty))}  <\infty.
\]

We shall show that 
$V(-t)u(t)$ is a Cauchy sequence in $H^1\cap H^{0,1}$. 
As in the proof of Lemma \ref{lem2}, for $t_0 \le s<t$, we have
\begin{eqnarray*}
\lefteqn{\|V(-t)u(t)-V(-s)u(s)\|_{H_x^1}}\\
&=&|\mu|
\left\|\int_s^tV(-\tau)\pt_x(|u|^{2\alpha}u)d\tau\right\|_{H_x^1}\\
&\lesssim&
\|\pt_x(|u|^{2\alpha}u)\|_{L_x^{\frac54}L_t^{\frac{10}{9}}(s,t)}
+\|\pt_x^2(|u|^{2\alpha}u)\|_{L_x^{\frac54}L_t^{\frac{10}{9}}((s,t))}\\
&\lesssim&
\|u\|_{S((s,t))}^{2\alpha}
(\|\pt_xu\|_{L_x^{\infty}L_t^2([t_0,\infty))}
+\|\pt_x^2u\|_{L_x^{\infty}L_t^2}([t_0,\infty)))\\
&\to&0\quad\text{as}\ s\to\infty.
\end{eqnarray*}
Let us turn to the estimate in $H^{0,1}$.
By (\ref{v}), 
\begin{eqnarray}
\lefteqn{\|x(V(-t)u(t)-V(-s)u(s))\|_{L_x^2}}\label{N4}\\
&=&
\|V(-t)J(t)u(t)-V(-s)J(s)u(s)\|_{L_x^2}\nonumber\\
&\lesssim&\|V(-t)v(t)-V(-s)v(s)\|_{L_x^2}\nonumber\\
& &+t\||u|^{2\alpha}u(t)\|_{L_x^2}+s\||u|^{2\alpha}u(s)\|_{L_x^2}.
\nonumber
\end{eqnarray}
By assumption and Lemma \ref{ks}, 
we have $\|u(t)\|_{L^\infty} = O(t^{-1/3})$.
Hence, together with the mass conservation (\ref{mass}), 
one sees that the last two terms in the right hand side of (\ref{N4}) 
vanish as $s,t\to \infty$.
Further, since $v$ satisfies (\ref{E:v}), Proposition \ref{inho} yields
\begin{align}
\lefteqn{\|V(-t)v(t)-V(-s)v(s)\|_{L_x^2}}\label{N5}\\
&\lesssim
\left\|\int_s^tV(-\tau)|u|^{2\alpha}\pt_xvd\tau\right\|_{L_x^2}
+\left\|\int_s^tV(-\tau)|u|^{2\alpha}ud\tau\right\|_{L_x^2}\nonumber\\
&\lesssim
\||u|^{2\alpha}\pt_xv\|_{L_x^{\frac54}L_t^{\frac{10}{9}}((s,t))}
+\||u|^{2\alpha}u\|_{L_t^1L_x^2((s,t))}\nonumber\\
&\lesssim
\|u\|_{S(s,t)}^{2\alpha}
\|\pt_xv\|_{L_x^{\infty}L_t^2([t_0,\infty))}
+s^{-\frac23\alpha+1} (\sup_{t\ge1} t^{\frac13}\|u(t)\|_{L^\infty})^{2\alpha}\|u_0\|_{ L^2}
\nonumber\\
&\to 0\quad\text{as}\ s\to\infty.
\nonumber
\end{align}
Plugging (\ref{N5}) 
to (\ref{N4}), 
we obtain
\begin{eqnarray*}
\|x(V(-t)u(t)-V(-s)u(s))\|_{L_x^2}\to0\quad\text{as}\ s\to\infty.
\end{eqnarray*}
Therefore we have that $V(-t)u(t)$ is a Cauchy sequence 
in $H^1\cap H^{0,1}$. 
This implies that $u(t)$ scatters in $H^1 \cap H^{0,1}$
for positive time direction.
\end{proof}

\begin{proof}[Proof of Theorem \ref{thm1}]  
Theorem \ref{thm1} is an immediate consequence 
of Lemmas \ref{lem3} and \ref{L:scatter1}. 
\end{proof}


\section{Proof of Theorem \ref{thm2}} 

In this section we prove Theorem \ref{thm2}. 

\begin{proof}[Proof of Theorem \ref{thm2}]
Let $u(t)$ be a maximal-lifespan $H^1 \cap H^{0,1}$-solution.

\subsubsection*{Step 1} Let us prove ``(iii)$\Rightarrow$(ii)''.
Suppose that for some 
$\kappa> \frac{\alpha}{3(\alpha-1)(2\alpha+1)}$ 
and $t_0 \in I_{\max}$, 
\[
	R:= \| u\|_{S([t_0,T_{\max}))} +\sup_{t\in [t_0,T_{\max})} \Jbr{t}^{\kappa} \| u(t)\|_{L^{2(2\alpha+1)}_x}
	 <\infty.
\]
By Theorem \ref{T:bu}, we see that $T_{\max}=\infty$. 
Further, by Corollary \ref{C:h1}, we obtain
\[
	\| u\|_{L^\infty ((t_0,\infty);H^1_x)} <\infty.
\]

We claim that there exists 
$\tilde{\kappa}>\frac1{2\alpha+1}$ such that
\begin{equation}\label{E:eqpf1-0}
	\sup_{t \ge t_1}\Jbr{t}^{\tilde{\kappa}} 
	\| u(t) \|_{L^{2(2\alpha+1)}_x} <\infty
\end{equation}
for some $t_1\ge t_0$.
We consider the case $\kappa \le \frac1{2\alpha+1}$ 
since if $\kappa>\frac1{2\alpha+1}$ 
then this is trivial by choosing $\tilde{\kappa}=\kappa$.
Let us consider the case $\kappa<\frac1{2\alpha+1}$.
Let $\delta_0>0$ be a constant to be determined later. 
For any choice of $\delta_0>0$,
there exists $t_1\in I_{\max} \cap[ \max(t_0,1),\infty)$ such that
\[
	\|u\|_{S((t_1,\infty))} \le \delta_0.
\]
We apply Propositions \ref{ho} and \ref{inho}, 
and the assumption to obtain
\begin{align*}
	\|v\|_{Z((t_1,T))} \le{}& C\|v(t_1)\|_{L^2}
	+ C\|u\|_{S((t_1,T))}^{2\alpha} \|v\|_{Z((t_1,T))} 
	+ \| |u|^{2\alpha} u \|_{L^1_t L^2_x((t_1,T))}  \\
	\le{}& C\|v(t_1)\|_{L^2}
	+ C\|u\|_{S((t_1,T))}^{2\alpha} \|v\|_{Z((t_1,T))} + \| t^{-(2\alpha+1)\kappa} \|_{L^1_t(t_1,T)} R^{2\alpha+1} \\
	\le{}& C\|v(t_1)\|_{L^2}
	+ C \delta_0^{2\alpha}   \|v\|_{Z((t_1,T))}
	+ CR^{2\alpha+1} T^{1-(2\alpha+1)\kappa},
\end{align*}
where $Z(I)$ is as in \eqref{E:ZT}.
We choose $\delta_0$ so that $C\delta_0^{2\alpha} \le \frac12$. Then, we see that
\[
	\|v\|_{Z((t_1,T))} \le 2C\|v(t_1)\|_{L^2}
	+ 2CR^{2\alpha+1} T^{1-(2\alpha+1)\kappa}
\]
for any $T\in (t_1,\infty)$.
In particular, 
we obtain
\begin{equation}\label{E:eqpf1-1}
	\| v(t) \|_{L^2} \lesssim t^{1-(2\alpha+1)\kappa}
\end{equation}
for all $t\ge t_1 (\ge 1)$.
One then sees from this inequality,
Lemma \ref{ks}, and the mass conservation (\ref{mass}) that
\begin{align*}
	t^{\frac{2\alpha}{3(2\alpha+1)}} \|u(t)\|_{L^{2(2\alpha+1)}}
	\le \|u_0\|_{L^2}^{\frac{\alpha+1}{2\alpha+1}} (\|v(t)\|_{L^2}+ t\|u(t)\|_{L^{2(2\alpha+1)}}^{2\alpha+1} )^{\frac{\alpha}{2\alpha+1}} 
	\lesssim t^{\frac{\alpha}{2\alpha+1} -\alpha \kappa}.
\end{align*}
Hence,
\begin{equation}\label{E:eqpf1-2}
	t^{\kappa_1} \| u(t) \|_{L^{2(2\alpha+1)}_x} \lesssim 1
\end{equation}
for all $t\ge t_1$, where
$\kappa_1:=\alpha \kappa-\frac{\alpha}{3(2\alpha+1)} $.
One sees that
\[
	\kappa_1 - \kappa = (\alpha-1) \( \kappa-\frac{\alpha}{3(\alpha-1)(2\alpha+1)}\) >0
\]
by assumption on $\kappa$.
This implies that \eqref{E:eqpf1-2} is a better decay estimate. 
If $\kappa_1<\frac{1}{2\alpha+1}$ then
we repeat the above argument starting with the better estimate \eqref{E:eqpf1-2}. Then, we obtain
\begin{equation*}
	t^{\kappa_2} \| u(t) \|_{L^{2(2\alpha+1)}_x} \lesssim 1
\end{equation*}
for all $t\ge t_1$, where $t_1$ is the exactly same one and
$\kappa_2:=\alpha \kappa_1-\frac{\alpha}{3(2\alpha+1)} $.
Similarly,
we construct $\kappa_j$ by induction. More precisely, if
$\kappa_{j}<\frac{1}{2\alpha+1}$
then we repeat the above argument to construct $\kappa_{j+1}>\kappa_j$ by $\kappa_{j+1}:=\alpha \kappa_j-\frac{\alpha}{3(2\alpha+1)} $.
Since
\begin{align*}
	\kappa_{j+1}-\kappa_j
	&= (\alpha-1) \( \kappa_j-\frac{\alpha}{3(\alpha-1)(2\alpha+1)}\) \\
	&>(\alpha-1) \( \kappa-\frac{\alpha}{3(\alpha-1)(2\alpha+1)}\)= \kappa_1-\kappa
\end{align*}
for every $j$,
we have $\kappa_j \ge \kappa + j(\kappa_1-\kappa)$. Hence,
this induction procedure stops at a finite time, i.e., there exists finite $j_0$ such that 
$\kappa_{j_0-1}< \frac1{2\alpha+1}$, 
$\kappa_{j_0}\ge \frac1{2\alpha+1}$, and
\begin{equation*}
	t^{\kappa_{j_0}} \| u(t) \|_{L^{2(2\alpha+1)}_x} \lesssim 1
\end{equation*}
holds for all $t\ge t_1$.
If $\kappa_{j_0}>\frac1{2\alpha+1}$ then we have \eqref{E:eqpf1-0} with the choice $\tilde{\kappa} = \kappa_{j_0}$.
Let us consider the case $\kappa_{j_0}=\frac1{2\alpha+1}$.
In this case, we replace
$\kappa_{j_0}$ by some number 
between $\frac{\alpha+3}{3\alpha(2\alpha+1)}$ and $\frac1{2\alpha+1}$, say
\[
	\kappa_{j_0}= \frac12\(\frac1{2\alpha+1} + \frac{\alpha+3}{3\alpha(2\alpha+1)}\),
\]
and apply the above argument once again. Then, one obtain \eqref{E:eqpf1-0} since
\[
	\kappa_{j_0+1}>\frac1{2\alpha+1} \Longleftrightarrow \kappa_{j_0} >\frac{\alpha+3}{3\alpha(2\alpha+1)}.
\]
The case $\kappa=\frac1{2\alpha+1}$ is handled also in this way.

With the estimate \eqref{E:eqpf1-0} in hand, we obtain a refined estimate for $v$.
Arguing as in the proof of \eqref{E:eqpf1-1}, we have
\begin{align*}
	\|v\|_{Z((t_1,T))} \le {}& C\|v(t_1)\|_{L^2}
	+ \frac12 \|v\|_{Z((t_1,T))} + C\| t^{-(2\alpha+1)\tilde{\kappa}} \|_{L^1((t_1,T))}  
\end{align*}
for any $T\ge t_1$. As
 $\tilde{\kappa}(2\alpha+1) > 1$, the third term in the right hand side is finite and bounded uniformly in $T$. Hence, we obtain
\begin{equation}\label{E:eqpf1-3}
	\| v \|_{L^\infty_t L^2_x([t_1,\infty))} < \infty.
\end{equation}
By combining $\|u\|_{L^\infty_t H^1_x ([t_0,\infty))} <\infty$,
\eqref{E:eqpf1-0}, and  \eqref{E:eqpf1-3}, one obtains
\[
		 \|V(-t)u\|_{L^\infty_t H^{0,1}_x([t_1,\infty))} \lesssim \|u\|_{L^\infty_t L^2_x([t_1,\infty))}+ \| Ju \|_{L^\infty_t L^2_x([t_1,\infty))}
< \infty.
\]
This is property (ii) since $t_1 \in I_{\max}$ and $T_{\max}=\infty$.
Thus, we complete the proof of ``(iii)$\Rightarrow$(ii)''.

\subsubsection*{Step 2}
We next prove ``(ii)$\Rightarrow$(i)''.
This part corresponds to the so-called conditional scattering.

Suppose that 
\begin{eqnarray}
\|V(-t)u\|_{L_t^{\infty}H_x^{0,1}(t_0,T_{\mathrm{max}})}<+\infty
\label{E:pf-gb}
\end{eqnarray}
for some $t_0 \in I_{\max}$.
By the local well-posedness (Theorem \ref{T:lwp}), 
one sees that $T_{\max}=\infty$.
Hence, by replacing $t_0$ with a larger one if necessary, 
one may suppose that $t_0>1$.
Let us prove 
the bound
\begin{equation}\label{E:eqpf2-1}
	\|u\|_{S((\widetilde{t}_0,\infty))} <\infty
\end{equation}
holds for some $\widetilde{t}_0\ge t_0$.
By Lemma \ref{ks}, the assumption \eqref{E:pf-gb} 
and the mass conservation (\ref{mass}), one sees that
\[
	t^{\frac{2\alpha-1}{3(2\alpha+1)}} 
	\|u(t) \|_{L^{2\alpha+1}} \lesssim
	\|u_0\|_{L^2}^{\frac12 
	+ \frac1{2\alpha+1}} \|J(t)u
	\|_{L^\infty L^2}^{\frac12 - \frac1{2\alpha+1}} <\infty.
\]
In particular, $\|u(t) \|_{L^{2\alpha+1}}$ is bounded uniformly in $t$.
Then, by the energy conservation (\ref{energy}) and the assumption 
\eqref{E:pf-gb} on $u$, we have
\begin{equation}\label{E:eqpf2-2}
\sup_{t\in [t_0,\infty)} 
\|V(-t)u(t)\|_{H^1 \cap H^{0,1}} <\infty.
\end{equation}
Pick a time sequence $\{t_n\}_{n\ge1} \subset [t_0,\infty)$ 
so that $t_n<t_{n+1}\to \infty$ as $n\to\infty$.
Then, by means of Lemma \ref{L:tb},
one can choose a subsequence, 
which we denote again by $\{t_n\}$, so that
$V(-t_n)u(t_n)$ converges (strongly) in $\hat{L}^\alpha$.
Let $\psi_+ \in \hat{L}^\alpha$ be the limit of the subsequence.


We 
let $\tilde{u}(t)$ be a unique $\hat{L}^\alpha$-solution to \eqref{gKdV}
which scatters to $\psi_+$ in $\hat{L}^\alpha$, i.e.,
\[
	\| V(-t)\tilde{u}(t) - \psi_+\|_{\hat{L}^\alpha}
	\to 0
\]
as $t\to \infty$.
We choose $T \in \R$ so that $\tilde{u}(t)$ exists on $[T,\infty)$.
Without loss of generality, we may suppose that $T \ge t_0$.

Note that
\begin{eqnarray}
	M:= \norm{\tilde{u}}_{S([T,\infty))\cap X([T,\infty))} \in [0,\infty).
	\label{u1}
\end{eqnarray}
Let $\eps=\eps(M)$ be the number given by 
long time perturbation 
(Proposition \ref{prop:long}). 
Our next goal is to show that
\[
	\norm{u}_{S([T,\infty))\cap X([T,\infty))} \le M+\eps.
\]
We apply Proposition \ref{prop:long} 
with the choice $\widetilde{u}(t):=\tilde{u}(t)$,
$I:=[T,\infty)$, and $t_0:=t_n$.
Note that $t_0 \in I$ for large $n$.
\eqref{ba1} is satisfied with the above $M$.
Further, since $\tilde{u}$ is a solution to \eqref{gKdV}, one has $\mathcal{E}=0$.
By Proposition \ref{ho},
\begin{align*}
	&\| V(t-t_n)(u(t_n) -\tilde{u}(t_n)) \|_{S([T,\infty)) \cap X([T,\infty))} \\
	&\le \| V(-t_n)u(t_n) -V(-t_n)\tilde{u}(t_n) \|_{\hat{L}^\alpha} \\
	&\le \| V(-t_n)u(t_n) -\psi_+ \|_{\hat{L}^\alpha} 
	+ \| \psi_+ -V(-t_n)\tilde{u}(t_n) \|_{\hat{L}^\alpha} \to 0
\end{align*}
as $n\to \infty$.
Hence, \eqref{ba2} is fulfilled for large $n$. Hence, one has
\[
	\| u - \tilde{u} \|_{S([T,\infty))\cap X([T,\infty))} \le \eps.
\]
We have the desired conclusion by combining this with \eqref{u1}.
Thus, we obtain \eqref{E:eqpf2-1} with the choice $\widetilde{t}_0=T$.
By means of Lemma \ref{L:scatter1}, \eqref{E:pf-gb} and \eqref{E:eqpf2-1} imply that
$u(t)$ scatters in $H^1 \cap H^{0,1}$ for positive time direction.
Thus, we completes the proof of ``(ii)$\Rightarrow$(i)''.


\subsubsection*{Step 3}
Let us finally prove ``(i)$\Rightarrow$(iii)''.

Suppose that a maximal-lifespan $H^1 \cap H^{0,1}$-solution $u$ scatters in $H^1 \cap H^{0,1}$ for positive time direction.
This immediately implies that
\begin{eqnarray}
	\lefteqn{\|u\|_{L^\infty_t H^1_x ([t_0,\infty))}+ 
	\|Ju\|_{L^\infty_t L^2_x ([t_0,\infty))}}
	\qquad\label{E:eqpf3-1}\\
	&\lesssim& \|V(-t)u(t)\|_{L^\infty_t (H^1_x\cap H^{0,1}_x) ([t_0,\infty))}
	 <\infty\nonumber
\end{eqnarray}
for any $t_0 \in I_{\max}$.
We fix $t_0>1$.
By the embedding $H^1 \cap H^{0,1} \hookrightarrow \hat{L}^\alpha$, we see that solution $u$ scatters also in $\hat{L}^\alpha$, which is equivalent to
\begin{equation}\label{E:eqpf3-2}
	\|u\|_{S([t_0,\infty))} <\infty.
\end{equation}
For $t\ge t_0(>1)$, one sees from Lemma \ref{ks} and \eqref{E:eqpf3-1} that
\begin{eqnarray}
t^{\frac{2\alpha}{3(2\alpha+1)}}
\|u(t)\|_{L^{2(2\alpha+1)}_x}
\lesssim 
\|u(t)\|_{L^2}	+ \|J(t)u(t)\|_{L^2}
\le C<\infty.\label{E:eqpf3-4}
\end{eqnarray}
Hence, combining \eqref{E:eqpf3-2} 
and \eqref{E:eqpf3-4},
we obtain
\[
	\|u\|_{S([t_0,\infty))} + \sup_{t\ge t_0} 
	\Jbr{t}^{\frac{2\alpha}{3(2\alpha+1)}}\|u(t)\|_{L^{2(2\alpha+1)}_x} < \infty,
\]
which is (iii).
Note that 
$\frac{2\alpha}{3(2\alpha+1)} > \frac{\alpha}{3(\alpha-1)(2\alpha+1)}$ if and only if $\alpha>3/2$.
This completes the proof of ``(i)$\Rightarrow$(iii)''.

Finally, suppose (i), (ii), and (iii) hold.
$T_{\max}=\infty$ follows, for instance, from (i).
We prove the bound. 
Since $V(-t)u(t)$ converges in $H^1 \cap H^{0,1}$ as $t\to\infty$ it is bounded in $H^1 \cap H^{0,1}$ uniformly in $t\in [t_0,\infty)$ for any $t_0 \in I_{\max}$.
This also implies the $L^\infty$-decay estimate 
\[
	\sup_{t\in [t_0,\infty)} \Jbr{t}^{\frac13} \|u(t)\|_{L^\infty_x} \lesssim 1
\]
by means of Lemma \ref{ks}.
The bound in $S([t_0,\infty))$ follows from (iii).
\end{proof}

\subsection*{Acknowledgements} 
S.M. was  supported by JSPS KAKENHI Grant Numbers JP23K20803, JP23K20805, and JP24K00529.
J.S. was supported by JSPS KAKENHI Grant Numbers 
JP21K18588 and JP23K20805.  


\end{document}